\documentclass{amsart}

\usepackage{array,amsmath,amsthm,amsxtra,amssymb}

\usepackage{supertabular}
\usepackage[absolute]{textpos}

\textblockorigin{0in}{0in}
\newtheorem{theorem}{Theorem}
\newtheorem{corollary}[theorem]{Corollary}
\newtheorem{lemma}[theorem]{Lemma}
\newtheorem{prop}[theorem]{Proposition}

\theoremstyle{definition}
\newtheorem{defn}[theorem]{Definition}

\newtheorem{remark}[theorem]{Remark}

\numberwithin{theorem}{section}

\newcommand{\hgq}[5]{
_{2}F_{1} \left( 
\begin{matrix}
#1, & #2 \\
   & #3 \\
\end{matrix}
\bigg{\vert} #4
\right)_#5 
}

\newcommand{\fp}
{\mathbb{F}_p}

\newcommand{\fq}
{\mathbb{F}_q}

\newcommand{\fqc}
{\mathbb{F}_q^{\times}}

\newcommand{\fphat}
{\widehat{\mathbb{F}_{p}^{\times}}}

\newcommand{\Z}{\mathbb{Z}}

\begin{document}

\pagenumbering{arabic}
\setlength{\headheight}{12pt}

\title[Traces of Hecke operators in level 1]{Traces of Hecke operators in level 1 and Gaussian hypergeometric functions}
\author{Jenny G. Fuselier}

\date{}

\begin{abstract}
We provide formulas for traces of $p^{th}$ Hecke operators in level 1 in terms of values of finite field $_2F_1$-hypergeometric functions, extending previous work of the author to all odd primes $p$, instead of only those $p \equiv 1 \pmod{12}$.  We first give a general level 1 trace formula in terms of the trace of Frobenius on a family of elliptic curves, and then we draw on recent work of Lennon to produce level 1 trace formulas in terms of hypergeometric functions for all primes $p > 3$.

\end{abstract}

\maketitle


\section{Introduction}

In recent years, relationships between traces of Hecke operators and counting points on families of varieties have been explored.  For example, Ahlgren and Ono \cite{AO00} described traces of $p^{th}$ Hecke operators in weight 4 and level 8 in terms of the number of $\fp$-points on a Calabi-Yau threefold, while in \cite{Ah02}, Ahlgren related traces of Hecke operators in weight 6 and level 4 to counting $\fp$-points on the Legendre family of elliptic curves.  The level 2 formula (for all weights) was made explicit in terms of the number of $\fp$-points on a family of elliptic curves by Frechette, Ono, and Papanikolas in \cite{FOP04}.  In \cite{Fu10}, the author considered the level 1 case and provided a formula in terms of the number of $\fp$-points on a one parameter family of elliptic curves for primes $p\equiv 1\pmod{12}$.  Most recently, Lennon \cite{Le11_2} considered the levels 3 and 9 scenarios.  Earlier work of Ihara \cite{Ih67} and Birch \cite{Bi68} gave reason to believe such formulas were possible.

Interestingly, these trace formulas also have a link to finite field hypergeometric functions introduced by Greene in the 1980s \cite{Gr87}.  Various authors \cite{AO2000a, FOP04, Le11_2} have used relations between values of Greene's hypergeometric functions and counting $\fp$-points on varieties to produce trace formulas in terms of hypergeometric functions.  In \cite[Thm. 1.2]{Fu10}, the author proved an explicit relationship between counting $\fp$-points on a one-parameter family of elliptic curves and the values of a particular $_2F_1$ function over $\fp$, which led to a level 1 trace formula in terms of hypergeometric functions. However, this formula was only proved for primes $p\equiv 1 \pmod{12}$.  Recently, Lennon \cite[Thms. 1.1 and 2.1]{Le11} has removed this restriction on the congruence class of $p$ to produce formulas that relate $\#E(\fq)$ to values of a $_2F_1$ function over $\fq$ for any $q=p^e$ where $q\equiv 1 \pmod{12}$.

In this paper, we provide a level 1 trace formula that holds for all $p>3$.  Then, we use Lennon's result to produce formulas for traces of Hecke operators in level 1 in terms of finite field hypergeometric functions.

\section{Statement of Main Results}

Let $p>3$ be prime and let $k\geq2$ be an even integer.  Define $F_k(x,y)=\frac{x^{k-1}-y^{k-1}}{x-y}.$ 
Then letting $x+y=s$ and $xy=p$ gives rise to polynomials $G_k(s,p)=F_k(x,y)$.  These polynomials can be written alternatively as
\begin{equation}\label{defn of G}
G_k(s,p)=\sum_{j=0}^{\frac{k}{2}-1} (-1)^j \binom{k-2-j}{j} p^j s^{k-2j-2}.
\end{equation}

Throughout, results will depend on the congruence class of $p$ mod 12.  As such, we set up some notation for various congruence classes of $p$ to be used throughout the remainder of the paper. Whenever $p \equiv 1 \pmod{4}$, we let $a,b\in\mathbb{Z}$ be such that  $p=a^2+b^2$ and $a+bi \equiv 1 \,(2+2i)$ in $\mathbb{Z}[i]$.  In that case, we define
\begin{equation}\label{mu}
\mu_k(p)=\frac{1}{2}[G_k(2a,p)+G_k(2b,p)].
\end{equation}

\noindent Similarly, whenever $p\equiv 1 \pmod{3}$, we let $c,d\in\mathbb{Z}$ be such that $p=c^2-cd+d^2$ and $c+d\omega \equiv 2 \,(3)$ in $\mathbb{Z}[\omega]$, where $\omega=e^{2\pi i/3}$.  This this case, we define
\begin{equation}\label{nu}
\nu_k(p)=\frac{1}{3}[G_k(c+d,p)+G_k(2c-d,p)+G_k(c-2d,p)] .
\end{equation}

We consider a one-parameter family of elliptic curves having $j$-invariant $\frac{1728}{t}$.  Specifically, for $t\in\fp \backslash \{0,1\}$, we let 
\begin{equation}\label{defn of E}
E_t: y^2=4x^3-\frac{27}{1-t}x - \frac{27}{1-t}.
\end{equation}
Let $a(t,p)$ denote the trace of the Frobenius endomorphism on $E_t$.  In particular, for $t\neq0,1$, we have $$a(t,p)=p+1-\#E_t(\fp). $$
Let $\Gamma=SL_2(\Z)$ and let $M_k$ and $S_k$, respectively, denote the spaces of modular forms and cusp forms of weight $k$ for $\Gamma.$  Further, let $\textnormal{Tr}_k(\Gamma,p)$ denote the trace of the Hecke operator $\textnormal{T}_k(p)$ on $S_k$.  Our first main result completely classifies the traces of cusp forms in level 1:

\begin{theorem}\label{trace formula}
Let $p>3$ be prime.  Then for even $k\geq4$,   
$$ \textnormal{Tr}_k(\Gamma,p)=-1-\lambda(k,p)-\sum_{t=2}^{p-1} G_k(a(t,p),p),$$ where
$$\lambda(k,p)=\begin{cases} 
\mu_k(p)+\nu_k(p)  & \text{if $p\equiv 1\pmod{12}$} \\
\mu_k(p)+ (-p)^{\frac{k}{2}-1} & \text{if $p\equiv 5\pmod{12} $}\\
\nu_k(p)+(-p)^{\frac{k}{2}-1} & \text{if $p\equiv 7\pmod{12}$} \\
2(-p)^{\frac{k}{2}-1} & \text{if $p\equiv 11\pmod{12} $}
\end{cases}$$
\end{theorem}

Next, we move to results which link these traces of Hecke operators in level 1 with hypergeometric functions over finite fields.  We begin with some preliminaries.  Let $p$ be a prime and let $q=p^e$.  Let  $\widehat{\mathbb{F}_{q}^{\times}}$ denote the group of all multiplicative characters on $\fqc$.   We extend $\chi\in\ \widehat{\mathbb{F}_{q}^{\times}}$ to all of $\fq$ by setting $\chi(0)=0$.   We let $\varepsilon$ denote the trivial character.  For $A,B\in\fphat$, let $J(A,B)$ denote the usual Jacobi symbol and define  

\begin{equation}\label{binom coeff}
\binom{A}{B}:=\frac{B(-1)}{q}J(A,\overline{B})=\frac{B(-1)}{q} \sum_{x\in \fq} A(x)\overline{B}(1-x).
\end{equation}

\noindent Greene defined \emph{hypergeometric functions over $\fq$} in the following way:   

\begin{defn}[\cite{Gr87} Defn. 3.10]\label{hg}
If $n$ is a positive integer, $x\in\fq$, and $A_0,A_1,\dots,A_n,$\\$B_1,B_2,\dots,B_n \in \widehat{\mathbb{F}_{q}^{\times}}$, then define
$$ _{n+1}F_{n} \left( 
\begin{matrix}
A_0, & A_1, & \dots, & A_n \\
     & B_1, & \dots, & B_n \\
\end{matrix}
\bigg{\vert} x
\right)_q 
:= \frac{q}{q-1} \sum_{\chi\in\widehat{\mathbb{F}_{q}^{\times}}} \binom{A_0\chi}{\chi} \binom{A_1\chi}{B_1\chi} \dots \binom{A_n\chi}{B_n\chi} \chi(x).$$
\end{defn}

In \cite[Thm. 1.2]{Fu10}, the author proved a formula giving an explicit relationship between $a(t,p)$ and a $_2F_1$ hypergeometric function over $\fp$, but required that $p\equiv 1\pmod{12}.$  In this case, the result of Theorem \ref{trace formula} can be rewritten to be in terms of a hypergeometric function over $\fp$.  However, \cite{Fu10} did not address the other classes of primes mod 12.  Notice that either $p\equiv 1\pmod{12}$ or, if not, then $p^2\equiv 1\pmod{12}$.  With this in mind, for the remainder of the paper we define $q=p^{e(p)}$, where
\begin{equation}\label{e(p)}
e(p)=\begin{cases} 1 &\text{if $p\equiv 1\pmod{12}$}\\ 2  & \text{if $p^2 \equiv 1\pmod{12}$}.\end{cases}
\end{equation}

We consider the same family of elliptic curves $E_t$, as defined in \eqref{defn of E}, but now over $\fq$, and with $a(t,q)=q+1-\#E_t(\fq)$.  Thanks to Lennon's results \cite{Le11} and an inverse pair given in \cite{Ri68}, we can now describe the traces of Hecke operators in level 1 in terms of a $_2F_1$ function over $\fq$ for the other classes of primes mod 12:

\begin{theorem}\label{trace with 2F1}
Let $p>3$ be prime such that $p=5,7,11 \pmod{12}$ with $q=p^{e(p)}$ and $T$ a generator of $\widehat{\mathbb{F}_{q}^{\times}}$.  Let $k\geq 4$ be even and $m=\frac{k}{2}-1$.  Define $H_m(x) :=\sum_{i=0}^{m}\binom{m+i}{m-i}x^i$.  Then
\begin{multline*}
\textnormal{Tr}_k(\Gamma,p)=-1-\lambda(k,p)\\
-\sum_{t=2}^{p-1}(-p)^mH_m\left(pT^{\frac{q-1}{2}}(2)T^{\frac{q-1}{4}}(1-t)\hgq{T^{\frac{q-1}{12}}}{T^{\frac{5(q-1)}{12}}}{\varepsilon}{t}{q}-2\right),
\end{multline*}
where $\lambda(k,p)$ is as in Theorem \ref{trace formula}.
\end{theorem}

Our final result is a generalization of \cite[Thm. 1.4]{Fu10}, giving a recursive formula for traces of Hecke operators in level 1 in terms of hypergeometric functions, now for all primes $p>3$:

\begin{theorem}\label{recursion all p}
Let $p>3$ be prime, and $q=p^{e(p)}$.  Let $k\geq 4$ be even, and $m=\frac{k}{2}-1$.  Further, let $T$ be a generator of $\widehat{\mathbb{F}_{q}^{\times}}$ and $b_i=p^{m-i}\left[\binom{2m}{m-i}-\binom{2m}{m-i-1}\right].$  Then
\begin{align*}
\textnormal{Tr}_{2(m+1)}(\Gamma,p)=&-1-\lambda(2m+2,p)+b_0(p-2)\\
&-\sum_{i=1}^{m-1}b_i\cdot(\textnormal{Tr}_{2i+2}(\Gamma,p)+1+\lambda(2i+2,p)),\\
& - \sum_{t=2}^{p-1} \left(\psi(t,q) \hgq{T^{\frac{q-1}{12}}}{T^{\frac{5(q-1)}{12}}}{\varepsilon}{t}{q}+2p(e(p)-1) \right)^{\frac{2m}{e(p)}} 
\end{align*}
where $$\psi(t,q)=-q T^{\frac{q-1}{2}}(2)T^{\frac{q-1}{4}}(1-t)$$
and  $\lambda(k,p)$ is as in Theorem \ref{trace formula}.
\end{theorem}


\section {Proof of Theorem \ref{trace formula}}\label{Trace formula proof}

To prove Theorem \ref{trace formula}, we begin with Hijikata's version of the Eichler-Selberg trace formula \cite{Hi89}.  The statement of this theorem requires some notation.  If $d<0$, $d\equiv 0,1 \pmod{4}$, let $\mathcal{O}(d)$ denote the unique imaginary quadratic order in $\mathbb{Q}(\sqrt{d})$ having discriminant $d$.  Let $h(d)=h(\mathcal{O}(d))$ be the order of the class group of $\mathcal{O}(d)$, and let $w(d)=w(\mathcal{O}(d))$ be half the cardinality of the unit group of $\mathcal{O}(d)$.  We then let $h^*(d)=h(d)/w(d)$.  The following theorem is the level 1 formulation of Hijikata's version of the Eichler-Selberg trace formula for any odd prime.

\begin{theorem}\label{hijikata new}
Let $p$ be an odd prime and $k\geq 2$ be even.  Then
$$\textnormal{Tr}_k(\Gamma,p)=-1-\frac{1}{2}\beta(p)(-p)^{\frac{k}{2}-1}-\sum_{0<s<2\sqrt{p}}G_k(s,p)\sum_{f|\ell}h^*\left(\frac{s^2-4p}{f^2}\right)+\delta(k),$$
where
\[
\beta(p)=
\begin{cases} h^*(-4p) & \text{ if $p\equiv 1 \pmod{4}$} \\
h^*(-4p)+h^*(-p) & \,\, \text{if $p\equiv 3\pmod{4}$,}
\end{cases}
\]

\noindent $\delta(k)=p+1$ if $k=2$ and $0$ otherwise, and where we classify integers $s$ with $s^2-4p<0$ by some positive integer $\ell$ and square-free integer $m$ via 
$$s^2-4p=\begin{cases}\ell^2m, & 0>m\equiv 1 \pmod{4}\\\ell^24m, & 0>m\equiv 2,3 \pmod{4}.\end{cases}$$

\end{theorem}

To link Theorem \ref{hijikata new} to $a(t,p)$, we need to consider all isomorphism classes of elliptic curves over $\fp$.  If $E$ is any elliptic curve defined over $\fp$, let $a(E)=p+1-\#E(\fp)$. Additionally, for a perfect field $K$, we define $$\emph{Ell}_K:=\{[E]_K : E \,\textnormal{is defined over}\, K\}$$
where $[E]_K$ denotes the isomorphism class of $E$ over $K$ and $[E_1]_K=[E_2]_K$ if there exists an isomorphism $\beta:E_1 \rightarrow E_2$ over $K$. We first address the cases  $j(E)=1728$ and $j(E)=0$.

\begin{lemma}\label{j=1728}
Let $p$ be an odd prime. Whenever $p \equiv 1 \pmod{4}$, define $a,b\in\mathbb{Z}$ be such that  $p=a^2+b^2$ and $a+bi \equiv 1 \,(2+2i)$ in $\mathbb{Z}[i]$.  Then, for $n\geq 2$ even, 

$$\sum_{\substack{[E]_{\fp}\in Ell_{\fp}\\j(E)=1728}}a(E)^n=\begin{cases} 2^{n+1}(a^n+b^n) & \text{if $p\equiv 1 \pmod{4}$}\\ 0 &\text{if $p\equiv 3 \pmod{4}.$} \end{cases}$$

\end{lemma}

\begin{proof}
The case $p\equiv 1 \pmod{4}$ was proved by the author in \cite[Lemma IV.3.3]{Fu07}.  If $p \equiv 3 \pmod{4}$ and $[E]_{\fp}\in Ell_{\fp}$, then $\#E(\fp)=p+1$, so $a(E)=0$.
\end{proof}

\begin{lemma}\label{j=0}
Let $p>3$ be prime. Whenever $p\equiv 1 \pmod{3}$, we let $c,d\in\mathbb{Z}$ be such that $p=c^2-cd+d^2$ and $c+d\omega \equiv 2 \,(3)$ in $\mathbb{Z}[\omega]$, where $\omega=e^{2\pi i/3}$.   Then, for $n\geq 2$ even, 
$$\sum_{\substack{[E]_{\fp}\in Ell_{\fp}\\j(E)=0}}a(E)^n=\begin{cases} 2[(c+d)^n+(2c-d)^n+(c-2d)^n] & \text{if $p\equiv 1 \pmod{3}$} \\ 0 & \text{if $p\equiv 2 \pmod{3}$}. \end{cases}$$ 
\end{lemma}

\begin{proof}
The case $p\equiv 1\pmod{3}$ was proved by the author in \cite[Lemma IV.3.5]{Fu07}.  If $p\equiv 2 \pmod{3}$  and $[E]_{\fp}\in Ell_{\fp}$, then $\#E(\fp)=p+1$, so $a(E)=0$.
\end{proof}

The proof of Theorem \ref{trace formula}  proceeds along the same line as the proof of the $p\equiv 1\pmod{12}$ case proved by the author in \cite{Fu10}.  In particular, we begin with the following extension of \cite[Lemma 5.3]{Fu10}.

\begin{lemma}\label{h to h*}
Let $p>3$ be prime.  Then for $n\geq 2$ even, 
\begin{multline*}
\sum_{0<s<2\sqrt{p}}s^n \sum_{f|\ell}h\left(\frac{s^2-4p}{f^2}\right)=\sum_{0<s<2\sqrt{p}}s^n \sum_{f|\ell}h^*\left(\frac{s^2-4p}{f^2}\right)\\+\frac{1}{4}\sum_{\substack{[E]_{\fp}\in Ell_{\fp}\\j(E)=1728}}a(E)^n+\frac{1}{3}\sum_{\substack{[E]_{\fp}\in Ell_{\fp}\\j(E)=0}}a(E)^n.
\end{multline*}
\end{lemma}

\begin{proof}
The proof for primes $p\equiv 1\pmod{12}$ is provided in \cite[Lemma 5.3]{Fu10}.  It can be adapted to hold for all $p>3$ once one verifies that the following two identities remain true:
\begin{equation}\label{h=-4}
\sum_{0<s<2\sqrt{p}}s^n\sum_{\substack{f|\ell\\\frac{s^2-4p}{f^2}=-4}}1 = \frac{1}{2}\sum_{\substack{[E]_{\fp}\in\emph{Ell}_{\fp}\\j(E)=1728}}a(E)^n
\end{equation}

\begin{equation}\label{h=-3}
\sum_{0<s<2\sqrt{p}}s^n\sum_{\substack{f|\ell\\\frac{s^2-4p}{f^2}=-3}}1 = \frac{1}{2}\sum_{\substack{[E]_{\fp}\in\emph{Ell}_{\fp}\\j(E)=0}}a(E)^n
\end{equation}

First consider \eqref{h=-4}.  The proof given in \cite{Fu10} holds for all primes $p \equiv 1 \pmod{4}$.  In light of Lemma \ref{j=1728}, we must verify that if $p\equiv 3\pmod{4}$, then 
$$\sum_{0<s<2\sqrt{p}}s^n\sum_{\substack{f|\ell\\\frac{s^2-4p}{f^2}=-4}}1=0.$$
We verify this by proving that no $s,f$ exist to contribute to the sums.  For, suppose $s,f\in\mathbb{Z}$ such that $0<s<2\sqrt{p}$ and $\displaystyle \frac{s^2-4p}{f^2}=-4.$  Then $4|s^2$, so $s$ must be even.  Substituting $s=2r$ and rearranging gives $r^2+f^2=p$, which is not possible since $p\equiv 3\pmod{4}$.  This verifies \eqref{h=-4} for the remaining primes.

We handle \eqref{h=-3} in a similar way.  The proof in \cite{Fu10} verifies the equation for $p \equiv 1 \pmod{3}$.  Keeping in mind Lemma \ref{j=0}, we must prove that  if $p\equiv 2\pmod{3}$, then $$\sum_{0<s<2\sqrt{p}}s^n\sum_{\substack{f|\ell\\\frac{s^2-4p}{f^2}=-3}}1=0.$$
Suppose then that we have $s,f\in\mathbb{Z}$ such that $0<s<2\sqrt{p}$ and $\displaystyle \frac{s^2-4p}{f^2}=-3.$ Then $4p=3f^2+s^2$ and hence $p\equiv 2\equiv s^2 \pmod{3}$.  However, since $\left(\frac{2}{3}\right)=-1,$ this is impossible.  This verifies \eqref{h=-3} for the remaining primes, and completes the proof of the lemma.
\end{proof}

The following proposition generalizes \cite[Prop. 5.4]{Fu10} by removing the restriction on the congruence class of $p \pmod{12}$.

\begin{prop}\label{main prop}
Let $p>3$ be prime and $n\geq 2$ be even. Then
$$\sum_{t=2}^{p-1}a(t,p)^n=\sum_{0<s<2\sqrt{p}}s^n \sum_{f|\ell}h^{*} \left(\frac{s^2-4p}{f^2}\right)\\-\alpha(n,p)-\gamma(n,p),$$
where
$$\alpha(n,p)=\begin{cases} 2^{n-1}(a^n+b^n) & \text{if $p\equiv 1\pmod{4}$}\\ 0 & \text{if $p\equiv 3\pmod{4}$},\end{cases}$$

$$ \gamma(n,p)\begin{cases}  \frac{1}{3}[(c+d)^n+(2c-d)^n+(c-2d)^n]& \text{if $p\equiv 1\pmod{3}$}\\ 0 & \text{if $p\equiv 2\pmod{3}$},\end{cases}$$
and $a,b,c,d$ are defined as in the statements of Lemmas \ref{j=1728} and \ref{j=0}.

\end{prop}


\begin{proof}
The bulk of the proof the author provides for the $p\equiv 1\pmod{12}$ case in \cite[Prop. 5.4]{Fu10} holds for the other congruence classes of $p$, so we give an outline here.  Since $j(E_t)=\frac{1728}{t}$, we have that 
\begin{align*}
\sum_{t=2}^{p-1}a(t,p)^n&=\sum_{\substack{[E]_{\overline{\mathbb{F}}_p}\in Ell_{\overline{\mathbb{F}}_p};\, E/\fp\\j(E)\neq 0,1728}}a(E)^n=\frac{1}{2}\sum_{\substack{[E]_{\fp}\in Ell_{\fp}\\j(E)\neq 0,1728}}a(E)^n\\
&=\frac{1}{2}\Biggl[ \sum_{[E]_{\fp}\in Ell_{\fp}}a(E)^n - \sum_{\substack{[E]_{\fp}\in Ell_{\fp}\\j(E)=1728}}a(E)^n-\sum_{\substack{[E]_{\fp}\in Ell_{\fp}\\j(E)=0}}a(E)^n\Biggr].
\end{align*}
Regardless of the congruence class of $p \pmod{12}$, the first sum in the last line above can still be written in terms of class numbers by combining Hasse's theorem with a theorem of Schoof \cite[Thm. 4.6]{Sc87}.  This results in 
\begin{align*}\label{h step}
\sum_{t=2}^{p-1}a(t,p)^n&=\sum_{0<s<2\sqrt{p}}s^n\sum_{f|\ell} h\left(\frac{s^2-4p}{f^2}\right)-\frac{1}{2}\sum_{\substack{[E]_{\fp}\in \emph{Ell}_{\fp}\\j(E)=1728}}a(E)^n-\frac{1}{2}\sum_{\substack{[E]_{\fp}\in \emph{Ell}_{\fp}\\j(E)=0}}a(E)^n\\
&=\sum_{0<s<2\sqrt{p}}s^n\sum_{f|\ell} h^*\left(\frac{s^2-4p}{f^2}\right)-\frac{1}{4}\sum_{\substack{[E]_{\fp}\in \emph{Ell}_{\fp}\\j(E)=1728}}a(E)^n-\frac{1}{6}\sum_{\substack{[E]_{\fp}\in \emph{Ell}_{\fp}\\j(E)=0}}a(E)^n,
\end{align*}
by Lemma \ref{h to h*}.  One now applies Lemmas \ref{j=1728} and \ref{j=0} in each appropriate congruence class to obtain the result.
\end{proof}

\noindent With the these tools in place, we now complete the proof of Theorem \ref{trace formula}.

\begin{proof}[Proof of Theorem \ref{trace formula}]
The proof proceeds in a similar fashion to the author's proof of the $p\equiv 1 \pmod{12}$ case in \cite[Thm. 1.3]{Fu10}, with a few modifications.  We still begin with an application of Theorem \ref{hijikata new} and then substitute the definition of $G_k(s,p)$.  This gives
\begin{align}
\textnormal{Tr}_k(\Gamma,p)       &=-1-\frac{1}{2}\beta(p)(-p)^{\frac{k}{2}-1}-(-p)^{\frac{k}{2}-1}\sum_{0<s<2\sqrt{p}}1\sum_fh^*\left(\frac{s^2-4p}{f^2}\right)\\
                   &\hspace*{0.2in}-\sum_{j=0}^{\frac{k}{2}-2}(-1)^j\binom{k-2-j}{j}p^j\sum_{0<s<2\sqrt{p}}s^{k-2j-2}\sum_fh^*\left(\frac{s^2-4p}{f^2}\right).\notag
\end{align}
Now, notice that the $k=2$ case of Theorem \ref{hijikata new} gives 
\begin{equation*}
0=p-\frac{1}{2}\beta(p)-\sum_{0<s<2\sqrt{p}}1\sum_fh^*\left(\frac{s^2-4p}{f^2}\right).
\end{equation*}
Substituting and applying Proposition \ref{main prop} with $n=k-2j-2$ gives
\begin{align}
\textnormal{Tr}_k(\Gamma,p)                 &=-1+(-p)^{\frac{k}{2}-1}\cdot(-p)-\sum_{j=0}^{\frac{k}{2}-2}(-1)^j\binom{k-2-j}{j}p^j\sum_{t=2}^{p-1}a(t,p)^{k-2j-2}\notag \\
                   &\hspace*{0.2in}-\sum_{j=0}^{\frac{k}{2}-2}(-1)^j\binom{k-2-j}{j}p^j\alpha(k-2j-2,p)\label{first step} \\
                   &\hspace*{0.2in}-\sum_{j=0}^{\frac{k}{2}-2}(-1)^j\binom{k-2-j}{j}p^j\gamma(k-2j-2,p).\notag
\end{align}
To complete the proof, we distribute the copies of $(-p)^{\frac{k}{2}-1}$ to the three summations in a specific way.  Notice that
$$(-p)^{\frac{k}{2}-1}(-p)=-(-p)^{\frac{k}{2}-1}(p-2)-(-p)^{\frac{k}{2}-1}-(-p)^{\frac{k}{2}-1}.$$
First, since  $G_2=1$, we see that 
\begin{equation}\label{sum 1}
-(-p)^{\frac{k}{2}-1}(p-2)-\sum_{j=0}^{\frac{k}{2}-2}(-1)^j\binom{k-2-j}{j}p^j\sum_{t=2}^{p-1}a(t,p)^{k-2j-2}=-\sum_{t=2}^{p-1}G_k(a(t,p),p).
\end{equation}

A straightforward calculation for each of the congruence classes of $p \pmod{12}$ verifies 
\begin{align*}
\lambda(k,p)&=(-p)^{\frac{k}{2}-1} + \sum_{j=0}^{\frac{k}{2}-2}(-1)^j\binom{k-2-j}{j}p^j\alpha(k-2j-2,p)\\
&\quad +(-p)^{\frac{k}{2}-1}+ \sum_{j=0}^{\frac{k}{2}-2}(-1)^j\binom{k-2-j}{j}p^j\gamma(k-2j-2,p).
\end{align*}
\end{proof}


\section{Trace formulas in terms of hypergeometric functions}\label{traces and 2f1s}

We now prove Theorems \ref{trace with 2F1} and \ref{recursion all p}.  As mentioned before, one essential tool is a theorem of Lennon, which writes the trace of Frobenius of any elliptic curve in Weierstrass form in terms of a finite field hypergeometric function:

\begin{theorem}\cite[Thm. 2.1]{Le11}\label{cathy 2.1}
Let $q=p^e$, where $p>3$ is prime and $q\equiv 1 \pmod{12}$.  Let $E:y^2=x^3+ax+b$ be an elliptic curve over $\fq$ in Weierstrass form with $j(E)\neq 0,1728$.  Then the trace of the Frobenius map on $E$ can be expressed as
$$a(E(\fq))=-q\cdot T^{\frac{q-1}{4}}\left(\frac{a^3}{27}\right) \,\cdot \, \hgq{T^{\frac{q-1}{12}}}{T^{\frac{5(q-1)}{12}}}{T^{\frac{q-1}{2}}}{-\frac{27b^2}{4a^3}}{q}.$$
\end{theorem}

We now specify this theorem to our family of curves.

\begin{corollary}\label{cathy cor}
Let $p>3$ be prime and $q=p^{e(p)}$, where $e(p)$ is defined as in \eqref{e(p)}.  Then
$$a(t,q)=-qT^{\frac{q-1}{2}}(2)T^{\frac{q-1}{4}}(1-t) \hgq{T^{\frac{q-1}{12}}}{T^{\frac{5(q-1)}{12}}}{\varepsilon}{t}{q}.$$
\end{corollary}

To prove this corollary, we require a transformation law proved by Greene:

\begin{theorem}\cite[Thm 4.4(i)]{Gr87}
If  $A,B,C \in \widehat{\mathbb{F}_{q}^{\times}}$ and $x\in\fq \backslash \{0,1\}$, then
$$\hgq{A}{B}{C}{x}{q}=A(-1)\hgq{A}{B}{AB\overline{C}}{1-x}{q}.$$
\end{theorem}

\begin{proof}[Proof of Cor. \ref{cathy cor}]
After putting $E_t$ into Weierstrass form, we have $a=b=\frac{-27}{4(1-t)}$ in Theorem \ref{cathy 2.1}.  Then $\frac{a^3}{27}=\frac{-3^6}{4^3(1-t)^3}$ and $-\frac{27b^2}{4a^3}=1-t$.  Combining these simplifications with Greene's theorem above gives
\begin{equation}\label{a(t,q)}
a(t,q)=-qT^{\frac{q-1}{4}}\left(\frac{-3^6}{4^3(1-t)^3}\right)T^\frac{q-1}{12}(-1)  \hgq{T^{\frac{q-1}{12}}}{T^{\frac{5(q-1)}{12}}}{\varepsilon}{t}{q}.
\end{equation}
Now, using multiplicativity and the fact that $T$ has order $q-1$, we have
\begin{align*}
T^{\frac{q-1}{4}}\left( \frac{-3^6}{4^3(1-t)^3} \right)&=T^{\frac{3(q-1)}{4}}\left(\frac{-9}{4(1-t)} \right)=T^{\frac{q-1}{4}}\left( \frac{4(1-t)}{-9}\right) \\
&=T^{\frac{q-1}{2}}(2)T^{\frac{q-1}{4}}(1-t)T^{\frac{q-1}{4}}(-1)T^{\frac{q-1}{2}}(3)\\
&=T^{\frac{q-1}{2}}(2)T^{\frac{q-1}{4}}(1-t)T^{\frac{q-1}{4}}(-1),
\end{align*}
since $T^{\frac{q-1}{2}}$ is its own inverse and $q\equiv 1\pmod{12}$.  The proof is completed by making this substitution for $T^{\frac{q-1}{4}}\left( \frac{-3^6}{4^3(1-t)^3} \right)$ into \eqref{a(t,q)} and noting that $T^{\frac{q-1}{4}}(-1)T^\frac{q-1}{12}(-1)=T^{\frac{q-1}{3}}(-1)=1,$ since $-1=(-1)^3$ and $T^{\frac{q-1}{3}}$ has order 3.

\end{proof}

\begin{remark}
If $e(p)=1$ (i.e. $q=p$), the above corollary precisely matches the author's result \cite[Thm. 1.3]{Fu10}.
\end{remark}

\begin{remark}
Lennon gives another way of writing the trace of Frobenius in terms of a $_2F_1$ function in \cite[Thm. 1.1]{Le11}, using the $j$-invariant and discriminant of $E$.  We use Lennon's Theorem 2.1 because it leads to a simpler hypergeometric function in this instance.
\end{remark}

We require two more tools to prove our trace theorems.  First, note that whenever $e(p)=2$ (i.e. $q=p^2$), Theorem \ref{cathy 2.1} relates $a(t,p^2)$ to a hypergeometric function over $\mathbb{F}_{p^2}$. Even though our trace formula Theorem \ref{trace formula} is in terms of $a(t,p)$, we can still gain new information, since 
\begin{equation}\label{frob link}
a(t,p)^2=a(t,p^2)+2p.
\end{equation}

The last tool is an inverse pair given in \cite{Ri68}.   As in the statement of Theorems \ref{trace with 2F1} and \ref{recursion all p}, we let $m=\frac{k}{2}-1$ and also define $H_m(x) :=\sum_{i=0}^{m}\binom{m+i}{m-i}x^i$.  Then, as in \cite{Fu10}, notice
\begin{equation}\label{G and H}
G_k(s,p)=(-p)^mH_m\left(\frac{-s^2}{p}\right).
\end{equation}
Consider the inverse pair \cite[p. 67]{Ri68} given by
\begin{equation}\label{inverse pair}
\rho_n(x)=\sum_{k=0}^{n}\binom{n+k}{n-k}x^k, \,\, x^n=\sum_{k=0}^n(-1)^{k+n}\left[\binom{2n}{n-k}-\binom{2n}{n-k-1}\right]\rho_k(x).
\end{equation}
Applying this to the definition of $H_m$, we see
\begin{equation*}
x^m=\sum_{i=0}^{m}(-1)^{i+m}\left[\binom{2m}{m-i}-\binom{2m}{m-i-1}\right]H_i(x).
\end{equation*}
By combining \eqref{G and H} with the choice $x=\frac{-s^2}{p}$, we have
\begin{equation}\label{s2m eqn}
s^{2m}=\sum_{i=0}^{m}b_iG_{2i+2}(s,p), 
\end{equation}
where $b_i=p^{m-i}\left[\binom{2m}{m-i}-\binom{2m}{m-i-1}\right]$.  We may now prove Theorems  \ref{trace with 2F1} and \ref{recursion all p}.
\begin{proof}[Proof of Theorem \ref{trace with 2F1}]
Recall that in the statement of this theorem, $p\equiv 5, 7, 11 \pmod{12}$, so $q=p^2$.  By \eqref{G and H} and \eqref{frob link}, we have 
\begin{align*}
G_k(a(t,p),p)&=(-p)^mH_m\left(\frac{-a(t,p)^2}{p}\right)\\
&=(-p)^mH_m\left(\frac{-a(t,p^2)}{p}-2 \right)\\
&=(-p)^mH_m\left( pT^{\frac{q-1}{2}}(2)T^{\frac{q-1}{4}}(1-t)\hgq{T^{\frac{q-1}{12}}}{T^{\frac{5(q-1)}{12}}}{\varepsilon}{t}{q} -2 \right),
\end{align*}
by Corollary \ref{cathy cor}.
Combining this with Theorem \ref{trace formula} completes the proof.
\end{proof}

\begin{proof}[Proof of Theorem \ref{recursion all p}]
Recall that $p>3$ is any prime and $q=p^{e(p)}$, where $e(p)$ is defined as in \eqref{e(p)}.  The proof begins along the same lines as the proof of \cite[Thm. 1.4]{Fu10}.  Beginning with Theorem \ref{trace formula}, we have
\begin{equation}\label{trace form 2m+2}
\textnormal{Tr}_{2(m+1)}(\Gamma,p)=-1-\lambda(2m+2,p)-\sum_{t=2}^{p-1}G_{2m+2}(a(t,p),p).
\end{equation}
Now, \eqref{s2m eqn} implies 
$$s^{2m}=\sum_{i=0}^{m}b_iG_{2i+2}(s,p)=G_{2m+2}(s,p)+\sum_{i=0}^{m-1}b_iG_{2i+2}(s,p).$$
We isolate $G_{2m+2}(s,p)$ and take $s=a(t,p)$.  Substituting into \eqref{trace form 2m+2} gives
\begin{align*}
\textnormal{Tr}_{2(m+1)}(\Gamma,p)&=-1-\lambda(2m+2,p)-\sum_{t=2}^{p-1}\Biggl( a(t,p)^{2m} -\sum_{i=0}^{m-1}b_iG_{2i+2}(a(t,p),p)\Biggr) \\
&=-1-\lambda(2m+2,p)-\sum_{t=2}^{p-1}a(t,p)^{2m}+b_0(p-2)\\
&	\hspace*{0.2in}+\sum_{i=1}^{m-1}b_i\sum_{t=2}^{p-1}G_{2i+2}(a(t,p),p)\\
&=-1-\lambda(2m+2,p)-\sum_{t=2}^{p-1}a(t,p)^{2m}+b_0(p-2)\\
&\hspace*{0.2in}-\sum_{i=1}^{m-1}b_i(\textnormal{Tr}_{2i+2}(\Gamma,p)+1+\lambda(2i+2,p)),
\end{align*} 
since $G_2=1$ and by again applying Theorem \ref{trace formula} for the last equality.  To complete the proof, we consider $a(t,p)^{2m}$.  If $e(p)=1$ (i.e. $q=p$) then using either Corollary \ref{cathy cor} or \cite[Thm. 1.2]{Fu10}, we have 
$$a(t,p)^{2m}=\left(-qT^{\frac{q-1}{2}}(2)T^{\frac{q-1}{4}}(1-t) \hgq{T^{\frac{q-1}{12}}}{T^{\frac{5(q-1)}{12}}}{\varepsilon}{t}{q}\right)^{2m},$$
so our result matches \cite[Thm. 1.4]{Fu10} in this case.  

If $e(p)=2$ (i.e. $q=p^2$), then \eqref{frob link} gives
\begin{align*}
a(t,p)^{2m}&=(a(t,p)^2)^m=(a(t,p^2)+2p)^m\\
&=\left(-qT^{\frac{q-1}{2}}(2)T^{\frac{q-1}{4}}(1-t) \hgq{T^{\frac{q-1}{12}}}{T^{\frac{5(q-1)}{12}}}{\varepsilon}{t}{q}+2p\right)^m,
\end{align*}
by Corollary \ref{cathy cor}.

The final statement of the theorem combines these two cases together, completing our proof.

\end{proof}

\end{document}